\documentclass[4pt]{article} 

\usepackage[utf8]{inputenc} 
\usepackage{geometry} 
\usepackage{graphicx} 

\usepackage{amsmath}
\usepackage{amsfonts}
\usepackage{dsfont}
\usepackage{mathrsfs}
\usepackage{latexsym}
\usepackage{graphicx}
\usepackage{amssymb}
\usepackage{float}
\usepackage{epsfig}
\usepackage{epstopdf}
\usepackage{color,xcolor}
\usepackage{booktabs} 
\usepackage{array} 
\usepackage{paralist} 
\usepackage{verbatim} 
\usepackage{subfig} 

\newtheorem{theorem}{Theorem}[section]
\newtheorem{lemma}[theorem]{Lemma}
\newtheorem{corollary}[theorem]{Corollary}

\newtheorem{conjecture}[theorem]{Conjecture}

\usepackage{fancyhdr} 
\usepackage{sectsty}
\allsectionsfont{\sffamily\mdseries\upshape} 

\usepackage[nottoc,notlof,notlot]{tocbibind} 
\usepackage[titles,subfigure]{tocloft} 

\title{More results on the distance (signless) Laplacian eigenvalues of graphs\thanks{Supported by National Natural Science Foundation of China (No. 11471121).}}
\author{ Jie Xue$^{a}$,~~Huiqiu Lin$^{b}$\thanks{Corresponding author. Email:~jie\_xue@126.com (J. Xue), huiqiulin@126.com (H. Lin).},~~~Kinkar Ch. Das$^{c}$,~ Jinlong Shu$^{a}$
\\
{\footnotesize $^a$Department of Computer Science and Technology, East China Normal University, Shanghai, PR China }\\
{\footnotesize $^b$Department of Mathematics, East China University of Science and Technology, Shanghai, PR China}\\
{\footnotesize $^c$Department of Mathematics, Sungkyunkwan University, Suwon, Republic of Korea}}
\date{} 

\begin{document}
\maketitle

\begin{abstract}
Let $G$ be a connected graph with vertex set $V(G)$ and edge set $E(G)$. Let $Tr(G)$ be the diagonal matrix of vertex transmissions of $G$ and $D(G)$ be the distance matrix of $G$.
The distance Laplacian matrix of $G$ is defined as $\mathcal{L}(G)=Tr(G)-D(G)$. The distance signless Laplacian matrix of $G$ is defined as $\mathcal{Q}(G)=Tr(G)+D(G)$. In this paper, we give a lower bound on the distance Laplacian spectral radius in terms of $D_1$, as a consequence, we show that $\partial_1^L(G)\geq n+\lceil\frac{n}{\omega}\rceil$ where $\omega$ is the clique number of $G$. Furthermore, we give some graft transformations, by using them, we characterize the extremal graph attains the maximum distance spectral radius in terms of $n$ and $\omega$. Moreover, we also give bounds on the distance signless Laplacian eigenvalues of $G$, and give a confirmation on a conjecture due to Aouchiche and Hansen \cite{AM3}.


\end{abstract}

\section{Introduction}
\hspace*{\parindent}All graphs considered here are simple,
undirected and connected. Let $G$ be a graph with vertex set $V(G)$
and edge set $E(G)$. For $u,v\in V(G)$, we denote by $u\thicksim v$
if $u$ is adjacent to $v$, and $u\not\thicksim v$ otherwise. Let
$N_{G}(v)$ be the neighborhood of $v$ and $d(v)=|N_{G}(v)|$ be the
degree of $v$. We denote by $\delta(G)$ and $\Delta(G)$ the minimum
and maximum degrees of the vertices of $G$. The distance between
vertices $u$ and $v$, denoted by $d_{G}(u,v)$ or simply by $d(u,v)$,
is the length of a shortest path from $u$ to $v$. The \emph{diameter}
$\mathrm{diam}(G)$ is the maximum distance between any two vertices
of $G$. The \emph{Wiener index}
$W(G)$ of a connected graph $G$ is defined to be the sum of all
distances in $G$, i. e.,
$$W(G)=\frac{1}{2}\sum_{u,v\in V(G)}\,d(u,v).$$
Let $X$ and $Y$ be two subsets of vertices of a graph $G$. We denote by $E_{G}[X,Y]$ the set of edges of $G$ with one vertex in $X$ and the other in $Y$. Let $G[X]$ be the subgraph of $G$ induced by $X$. We denote by $G-e$ be the graph obtained from $G$ by deleting an edge $e\in E(G)$. In particular, we denote by $K_{n}-kK_{2}$ the graph obtained form a complete graph by deleting $k$ independent edges, where $1\leq k\leq \lfloor\frac{n}{2}\rfloor$.  The \emph{union} of graphs $G_{1}$ and $G_{2}$ is the graph $G_{1}\cup G_{2}$ with vertex set $V(G_{1})\cup V(G_{2})$ and edge set $E(G_{1})\cup E(G_{2})$. The \emph{join} of $G_{1}$ and $G_{2}$, denoted $G_{1}\vee G_{2}$, is the graph obtained from $G_{1}\cup G_{2}$ by adding edges joining every vertex of $G_{1}$ to every vertex of $G_{2}$. As usual, denoted by $P_{n}$, $C_{n}$, $K_{n}$ and $S_n$ the path, cycle, complete graph and star of order $n$. Let $S^+_n$ be the $n$-vertex unicyclic graph obtained by adding an edge to $S_n$.  The clique number $\omega(G)$ is the number of vertices of the largest clique of $G$. Let $K_{\omega}^{n-\omega}$ be the graph obtained from a clique $K_{\omega}$ and a path $P_{n-\omega}$ by adding an edge between an endvertex of the path and a vertex from the clique. The Tur\'{a}n graph $T_{n,\omega}$ is a complete $\omega$-partite graph on $n$ vertices whose parts have equal or almost equal sizes (that is, $\lceil \frac{n}{\omega}\rceil$ or $\lfloor \frac{n}{\omega}\rfloor$).

Let $G$ be a graph with vertex set $V(G)=\{v_{1},v_{2},\ldots,v_{n}\}$. The
distance matrix of $G$, denoted by $D(G)$, is the symmetric real
matrix with $(i,j)$-entry being $d(v_{i},v_{j})$. Up to now, the distance matrix has been most extensively studied. We refer the reader to the survey \cite{AM0} for more details about distance eigenvalues of graphs and their applications. For $v_i\in V(G)$, the
\emph{transmission} of $v_i$, denoted by $Tr(v_i)$ ($D_i$ or $D_{v_{i}}$), is the sum of distances from $v_i$ to all other vertices of $G$, that is,
$$Tr(v_i)=\sum_{v_j\in V(G)}d(v_i,v_j).$$
It is clear that $2W(G)=\sum\limits^n_{i=1}\, Tr(v_{i})$ and $Tr(v_i)$
is the sum of $i$-th row of $D(G)$. Let
$Tr(G)=\mathrm{diag}(Tr(v_1),Tr(v_2),\ldots,Tr(v_n))$ be the
diagonal matrix of vertex transmissions of $G$. Aouchiche and Hansen
\cite{AM} introduced the distance Laplacian matrix of a connected
graph $G$ as
$$\mathcal{L}(G)=Tr(G)-D(G).$$
Obviously, $\mathcal{L}(G)$ is a positive semi-definite symmetric matrix, and we denote its eigenvalues by
$\partial^L_1(G)\geq \partial^L_{2}(G)\geq \cdots\geq \partial^L_{n}(G)=0$. The distance signless Laplacian matrix of $G$ is defined as
$$\mathcal{Q}(G)=Tr(G)+D(G).$$
 Evidently, $\mathcal{Q}(G)$ is irreducible, non-negative, symmetric and positive semidefinite. Let $\partial^Q_1(G)\geq \partial^Q_2(G)\geq\cdots\geq \partial^Q_n(G)$
 denote the eigenvalues of $\mathcal{Q}(G)$. As usual, we call $\partial^L_1(G)$ and $\partial^Q_1(G)$ the distance Laplacian spectral radius and distance
 signless Laplacian spectral radius, respectively. Some more details about distance (signless) Laplacian spectral eigenvalues, see \cite{AM,AM1,AM3,AH}.

Recently, the distance Laplacian and distance signless Laplacian matrix of a graph has received increasing attention.  In \cite{AM1},  Aouchiche and Hansen first establish some properties of the distance Laplacian spectral eigenvalues. They also proposed some conjectures about some particular distance Laplacian
eigenvalues of a graph: (a) the multiplicity of the largest eigenvalue $\partial^L_1(G) \leq n-2$ if $G\ncong K_n$; (b) for any graph, the path is the unique graph with the maximum distance Laplacian spectral radius; (c) for unicyclic graph, $\partial^L_1(G)\geq \partial^L_1(S_{n}^{+})$ with equality if and only if $G\cong S_{n}^{+}$; (d) for a tree, $\partial^L_2(G)\geq 2n-1$ with equality if and only if  $G\cong S_{n}$; (e) for any graph, $\partial^L_2(G)\geq n$ with equality if and only if $G\cong K_{n}$ or $G\cong K_{n}-e$. Lin et al. \cite{LH} proved the conjecture (a), it is also resolved in \cite{SJ1} by other method. In \cite{SJ}, Marques da Silva and Nikiforov confirmed conjecture (b). (c), (d) and (e) were proved by Tian et al. \cite{TF}. Nath and Paul \cite{NM} considered the graphs with complement graph a tree or a unicyclic of order $n$ and characterized the graphs among them having the second smallest distance Laplacian eigenvalue $n+1$. Lin and Zhou \cite{LZ} determined the graph with minimum distance Laplacian spectral radius among connected graphs with fixed number of pendent vertices, the tree with minimum distance Laplacian spectral radius among trees with fixed bipartition, the graph with minimum distance Laplaicn spectral radius among graphs with fixed edge connectivity at most half of the number of vertices. Niu et al. \cite{NA} considered the graph with minimum distance Laplacian spectral radius among all connected bipartite graphs with a given matching number and a given vertex connectivity. Furthermore, Tian et al. \cite{TF1} considered a more general case, they determined the graph with minimum distance Laplacian spectral radius among general graphs with given matching number.

Aouchiche and Hansen \cite{AM1} described some elementary properties on distance signless Laplacian eigenvalues. Meanwhile, some conjectures about distance signless Laplacian eigenvalues are also proposed, for example: (f) for a tree, $\partial^Q_2(G)\geq \partial^Q_2(S_{n})$ with equality if and only if $G\cong S_{n}$; (g) for a unicyclic graph,  $\partial^Q_2(G)\geq \partial^Q_2(S_{n}^{+})$ with equality if and only if $G\cong S_{n}^{+}$. These two conjectures were proved by Das \cite{DK}. In \cite{TF}, Tian et al. also gave a proof of the conjecture (f). In \cite{XZ,XZL}, Xing et al. considerd the distance signless Laplacian spectral radius of some special graphs. Lin and Lu \cite{LH1} determined the extremal graphs with maximum and minimum distance signless Laplacian spectral radius among all connected graphs with a given clique number. In \cite{LZ1}, Lin and Zhou studied the effect of three types of graft transformations to decrease or increase the distance signless Laplacian spectral radius, by using these graft transformations, they determined some extremal graphs with the smallest distance signless Laplacian spectral radius among some special graph classes. In \cite{LH2}, Lin and Das determined all graphs with $\partial^Q_n(G)=n-2$ and all graphs with $\partial^Q_2(G)\in [n-2,n]$, they also gave a lower bound on $\partial^Q_2(G)$ with independence number and the extremal graph is also characterized.

In this paper, we study the distance Laplacian eigenvalues and distance signless Laplacian eigenvalues of connected graphs. In section 3, we give an improved lower bound for the distance Laplacian spectral radius. As a consequence, we characterize all connected graphs with distance Laplacian spectral radius not exceeding $n+2$. In section 4, we present two upper bounds for the distance Laplacian spectral radius. In section 5, we determine the graphs with maximum and minimum distance Laplacian spectral radii among all connected graphs with a given clique number. In section 6, first, we obtain a lower bound for distance signless Laplacian spectral radius. Then, an upper bound for the difference between the distance Laplacian spectral radius and the distance signless Laplacian spectral radius of a graph is presented. Finally, we consider the upper bounds for the smallest distance signless Laplacian eigenvalue. In the last section, we determined the graph with maximum distance signless Laplacian spectral radius among all connected unicyclic graphs, that prove a conjecture proposed by Aouchiche and Hansen \cite{AM3}.

\section{Preliminaries}
\hspace*{\parindent}In this section, we shall list some previously known results which are needed in the following sections. The following result is the well-known Cauchy interlacing theorem (see, for example, \cite{SC}).

\begin{lemma} {\bf(Cauchy interlacing theorem)}\label{Lem=r1}
Let $A$ be a Hermitian matrix with order $n$ and let $B$ be a principal submatrix of $A$ with order $m$. If $\lambda_1(A)\geq \lambda_2(A)\geq \cdots\geq \lambda_n(A)$ list the eigenvalues of $A$ and $\lambda_1(B)\geq \lambda_2(B)\geq \cdots\geq \lambda_m(B)$ list the eigenvalues of $B$, then
$$\lambda_{n-m+i}(A)\leq \lambda_{i}(B)\leq \lambda_{i}(A)$$
for $i =1,\ldots, m$.
\end{lemma}

Consider an $n\times n$ matrix
\[
M=\left(\begin{array}{ccccccc}
M_{1,1}&M_{1,2}&\cdots &M_{1,m}\\
M_{2,1}&M_{2,2}&\cdots &M_{2,m}\\
\vdots& \vdots& \ddots& \vdots\\
M_{m,1}&M_{m,2}&\cdots &M_{m,m}\\
\end{array}\right),
\]
whose rows and columns are partitioned according to a partitioning $Y_{1}, Y_{2},\ldots ,Y_{m}$ of $\{1,2,\ldots, n\}$. A \emph{quotient matrix} $R$ is the $m\times m$ matrix whose entries are the
average row sums of the blocks $M_{i,j}$ of $M$.

\begin{lemma}{\bf(\cite{HW})}\label{le10}
Let $M$ be a symmetric $n\times n$ matrix with eigenvalues $\lambda_{1}(M)\geq \cdots \geq \lambda_{n}(M)$ and $R$ be an $m\times m$ matrix with eigenvalues $\lambda_{1}(R)\geq \cdots \geq \lambda_{m}(R)$. Suppose that $R$ is a quotient matrix of partitioned matrix $M$. Then $\lambda_{1}(M)\geq \lambda_{1}(R)$.
\end{lemma}

\begin{lemma} {\bf(\cite{AM})} \label{zq1}
Let $G$ be a connected graph of order $n$ and $m~(\geq n)$ edges. Consider the connected graph $\tilde{G}$ obtained from $G$ by the
deletion of an edge. Let $\partial^L_1(G)\geq \partial^L_2(G)\geq \cdots\geq \partial^L_n(G)$ and
$\partial^L_1(\tilde{G})\geq \partial^L_2(\tilde{G})\geq \cdots\geq \partial^L_n(\tilde{G})$ denote the distance Laplacian eigenvalues of $G$ and
$\tilde{G}$, respectively. Then $\partial^L_i(\tilde{G})\geq \partial^L_i(G)$, for all $i=1,\,2,\ldots,\,n$.
\end{lemma}

A similar result about the distance singless Laplacian eigenvalues is also given.

\begin{lemma} {\bf(\cite{AM})} \label{zq2}
Let $G$ be a connected graph of order $n$ and $m~(\geq n)$ edges. Consider the connected graph $\tilde{G}$ obtained from $G$ by the
deletion of an edge. Let $\partial^Q_1(G)\geq \partial^Q_2(G)\geq \cdots\geq \partial^Q_n(G)$ and
$\partial^Q_1(\tilde{G})\geq \partial^Q_2(\tilde{G})\geq \cdots\geq \partial^Q_n(\tilde{G})$ denote the distance Laplacian eigenvalues of $G$ and
$\tilde{G}$, respectively. Then $\partial^Q_i(\tilde{G})\geq \partial^Q_i(G)$, for all $i=1,\,2,\ldots,\,n$.
\end{lemma}

\section{Lower bound on distance Laplacian spectral radius of graphs}

\hspace*{\parindent}In this section, we present some lower bounds on the distance Laplacian spectral radius. Let $G$ be a connected graph with transmissions $D_{1}\geq D_{2}\geq \cdots\geq D_{n}$.  Recently, Lin et al. \cite{LH} gave the following result.
\begin{lemma} {\bf(\cite{LH})} \label{k1}
Let $G$ be a connected graph of order $n$ and $D_1$ be the maximum transmission of $G$. Then
$$\partial^L_1(G)\geq D_1+\frac{D_1}{n-1}\,.$$
\end{lemma}

As a consequece, by $D_{1}\geq n-1$, they got a lower bound on the distance Laplacian spectral radius of graphs.

\begin{lemma} {\bf(\cite{LH})} \label{k2}
Let $G$ be a connected graph of order $n$. Then
$$\partial^L_1(G)\geq D_1+1\,,$$
equality holds if and only if $G\cong K_n$.
\end{lemma}

Here, by using the quotient matrix technique, we present a lower bound which improves the Lemma \ref{k2} when $G\ncong K_n$.

\begin{theorem}\label{th9}
Let $G~(\ncong K_n)$ be a connected graph with maximum transmission $D_{1}$. Then $$\partial^L_1(G)\geq D_{1}+2.$$ Moreover, if $\mathrm{diam}(G)\geq 3$, then $\partial^L_1(G)>D_{1}+2$.
\end{theorem}

\begin{proof}
Let $G$ be a connected graph with vertex set $\{v_{1},v_{2},\ldots, v_{n}\}$ and $\mathrm{diam}(G)\geq 2$. We may assume that $D_{v_{1}}=D_{1}$. Let $A=N_{G}(v_{1})$ and $B=\{v\in V(G) : d(v,v_{1})=2\}$.  Let $G'$ be the graph obtained from $G$ by adding  edges (if necessary) such that $G[A]$ is a clique and each vertex in $A$ is adjacent to each vertex in $B$. It is clear that $\partial^L_1(G)\geq \partial^L_1(G')$. Let $D'_{v_{1}}$ be the transmission of $v_{1}$ in $G'$. Note that $D'_{v_{1}}=D_{1}$. 
It suffices to show that $\partial^L_1(G')\geq D'_{v_{1}}+2$. Then $\{v_{1}\}, A, V(G')\backslash(\{v_{1}\}\cup A)$ is a partition of $V(G')$. Then we obtain a quotient matrix of $\mathcal{L}(G')$:
\begin{equation*}\begin{split}
R&=\left(\begin{array}{ccccccc}
D'_{v_{1}}&-|A|&-D'_{v_{1}}+|A|\\[1mm]
-1&D'_{v_{1}}-n+2&-D'_{v_{1}}+n-1\\[2mm]
\frac{|A|-D'_{v_{1}}}{n-|A|-1}&\frac{-|A|(D'_{v_{1}}-n+1)}{n-|A|-1}&\frac{|A|(D'_{v_{1}}-n+1)-(|A|-D'_{v_{1}})}{n-|A|-1}
\end{array}\right).
\end{split}\end{equation*}
By a direct calculation, we have
$$|(D'_{v_{1}}+2)I-R|=\frac{-n(D'_{v_{1}}+2)(D'_{v_{1}}-2n+|A|+2)}{n-|A|-1}.$$
Since $D'_{v_{1}}\geq |A|+2(n-|A|-1)=2n-|A|-2$, it follows that $|(D'_{v_{1}}+2)I-R|\leq 0$. Therefore, $\lambda_{1}(R)\geq D'_{v_{1}}+2$. From Lemma \ref{le10}, we have $\partial^L_1(G')\geq \lambda_{1}(R)\geq D'_{v_{1}}+2$, as desired.

So in the following, we may assume that $\mathrm{diam}(G)\geq 3$. If there exists a vertex $v\in V(G')$ such that $d_{G'}(v,v_{1})\geq 3$, then $D'_{v_{1}}>2n-|A|-2$, and so $\lambda_{1}(R)>D'_{v_{1}}+2$. This implies $\partial^L_1(G')>D'_{v_{1}}+2$. Otherwise, $V(G)=\{v_{1}\}\cup A\cup B$. Let $v$ be an arbitrary vertex in $A$. Then we have
$$\sum_{u\in V(G)\backslash A}d_{G}(u,v)=1+(n-|A|-1)+\varepsilon(v)=n-|A|+\varepsilon(v)$$
where $\varepsilon(v)$ is a non-negative integer corresponding to $v$. We may let $a=\sum\limits_{v\in A}\frac{\varepsilon(v)}{|A|}$. For the partitioning $\{v_{1}\}\cup A\cup B$ of $V(G)$, we get a quotient matrix of $\mathcal{L}(G)$:
\begin{equation*}\begin{split}
R^{*}&=\left(\begin{array}{ccccccc}
tr(v_{1})&-|A|&-tr(v_{1})+|A|\\[1mm]
-1&n-|A|+a&1-n+|A|-a\\[2mm]
\frac{|A|-tr(v_{1})}{n-|A|-1}&\frac{-|A|(n-|A|+a-1)}{n-|A|-1}&\frac{tr(v_{1})-|A|-|A|(1-n+|A|-a)}{n-|A|-1}
\end{array}\right).
\end{split}\end{equation*}
Since $D_{v_{1}}=2n-|A|-2$, by a simple calculation, we have
$$|(D_{v_{1}}+2)I-R^{*}|=\frac{-a|A|(2n-|A|)}{n-|A|-1}.$$
Clearly, $a\geq 0$. If $a>0$, then $|(D_{v_{1}}+2)I-R^{*}|<0$. Hence $\lambda_{1}(R^{*})>D_{v_{1}}+2$. It follows from Lemma \ref{le10} that $\partial^L_1(G)\geq \lambda_{1}(R^{*})>D_{v_{1}}+2$. If $a=0$, then $\varepsilon(v)=0$ for $v\in A$. This implies every vertex in $A$ is adjacent to every vertex in $V(G)\backslash A$, which contradicts $\mathrm{diam}(G)\geq 3$. This completes the proof.\hspace*{\fill}$\Box$
\end{proof}

Using Theorem \ref{th9}, we characterize all connected graphs with distance Laplacian spectral radius not exceeding $n+2$.
\begin{theorem}
Let $G$ be a connected graph of order $n$. Then we have the following statements.

\noindent{\rm{(1)}} $\partial_1^L(G)=n$ if and only if $G\cong K_n.$

\noindent{\rm{(2)}} If $G\ncong K_n$, then $\partial_1^L(G)\geq n+2$ with equality holding if and only if $G\cong K_n-kK_2$ where $1\leq k \leq \lfloor\frac{n}{2}\rfloor.$
\end{theorem}

\begin{proof}
According to Lemma \ref{k1}, the conclusion $(1)$ follows immediately. So in the following, we may assume that $G\ncong K_n$, then $\Delta(G^c)\geq 1$.
If $\Delta(G^c)\geq 2$, then $D_1\geq (n-3)+4=n+1$. By Theorem \ref{th9}, it follows that $\partial_1^L(G)\geq n+3.$
If $\Delta(G^c)=1$, then $G\cong K_n-kK_2$ where $1\leq k \leq \lfloor\frac{n}{2}\rfloor.$ Furthermore, it is easy to check that $\partial_1^L(K_n-kK_2)=n+2$. Thus conclusion (2) holds.\hspace*{\fill}$\Box$
\end{proof}

\section{Upper bound on distance Laplacian spectral radius of graphs}
\hspace*{\parindent}In this section, we present some upper bounds on the distance Laplacian spectral radius of a connected graph. The following two lemmas are needed.
\begin{lemma}{\bf(\cite{AM})}\label{lkl}
Let $G$ be a connected graph on $n\geq 3$ vertices. Then $\partial^L_i(G)\geq \partial^L_i(K_{n})=n$, for all $1\leq i\leq n-1$, and $\partial^L_n(G)= \partial^L_n(K_{n})=0$.
\end{lemma}

\begin{lemma}{\bf(\cite{TF})}\label{l2l}
For any graph $G$ on $n\geq 4$ vertices, $\partial^L_2(G)\geq n$ with equality if and only if $G\cong K_{n}$  or $G\cong K_{n}-e$ for some $e\in E(K_{n})$.
\end{lemma}

The next theorem gives a sharp upper bound on the distance Laplacian spectral radius of a graph.
\begin{theorem}\label{ul1}
Let $G$ be a connected graph on $n\geq 4$ vertices with Wiener index $W$. Then $$\partial^L_1(G)\leq 2W-n(n-2)$$ with equality holds if and only if $G\cong K_{n}$.
\end{theorem}
\begin{proof}
It is clear that
$$\sum_{i=1}^{n}\partial^L_i(G)=D_{1}+D_{2}+\cdots+D_{n}=2W.$$ Then,
by Lemmas \ref{lkl} and \ref{l2l}, we have $$\partial^L_1(G)\leq
2W-\partial^L_2(G)-\partial^L_3(G)-\cdots-\partial^L_{n-1}(G)\leq
2W-n(n-2).$$ Further, if the equality holds, then
$\partial^L_2(G)=n$, and so $G\cong K_{n}$  or $G\cong K_{n}-e$.
Recall that the distance Laplacian eigenvalues of $K_{n}$ is
$\{n^{n-1},0^{1}\}$ and the distance Laplacian eigenvalues of
$K_{n}-e$ is $\{n+2^{1},n^{n-2},0^{1}\}$ (see \cite{AM}). It is easy
to check that the equality holds only if $G\cong K_{n}$. Thus we complete the
proof.\hspace*{\fill}$\Box$
\end{proof}

Another upper bound for the distance Laplacian spectral radius of a graph is the following.

\begin{theorem}
Let $G$ be a connected graph of order $n$. Then
\begin{equation}
\partial^L_1(G)<D_1+\sqrt{2\,\sum\limits_{1\leq i<j\leq n}\,d^2_{ij}-\frac{1}{n}\,\sum\limits^n_{i=1}\,D^2_i}\,,\label{s1}
\end{equation}
where $D_i$ is the transmission of the vertex $v_i$ and $D_1\geq D_2\geq \cdots\geq D_n$.
\end{theorem}

\begin{proof}
Let $X=(x_1,\,x_2,\ldots,\,x_n)^t$ be a unit eigenvector
corresponding to $\partial^L_1(G)$ of the distance Laplacian matrix
$\mathcal{L}(G)$ of $G$. Then we have
$\mathcal{L}(G)X=\partial^L_1(G)\,X$\,. For $v_i\in V(G)$, we have
\begin{equation}
\partial^L_1(G)\,x_i=D_i\,x_i-\sum\limits^n_{k=1}\,d_{ik}\,x_k\,,\nonumber
\end{equation}
that is,
\begin{equation}
(\partial^L_1(G)-D_i)^2\,x^2_i=\Big(\sum\limits^n_{k=1}\,d_{ik}\,x_k\Big)^2\,.\label{s2}
\end{equation}
By Cauchy-Schwarz inequality, we have
\begin{eqnarray}
\left(\sum\limits^n_{k=1}\,x_k\left(d_{ik}-\frac{1}{n}\,\sum\limits^n_{k=1}\,d_{ik}\right)\right)^2&\leq
&\sum\limits^n_{k=1}\,x^2_k\,\sum\limits^n_{k=1}\,\left(d^2_{ik}+\frac{1}{n^2}\,\left(\sum\limits^n_{k=1}\,d_{ik}\right)^2-\frac{2}{n}\,d_{ik}\,\sum\limits^n_{k=1}\,d_{ik}\right)\label{s3}\\[5mm]
&=&\sum\limits^n_{k=1}\,d^2_{ik}-\frac{1}{n}\,\left(\sum\limits^n_{k=1}\,d_{ik}\right)^2~~~~~~~(\mbox{ since }~\sum\limits^n_{i=1}\,x^2_i=1).\nonumber
\end{eqnarray}
It is well known that $\sum\limits^n_{i=1}\,x_i=0$. Then it follows that
$$\left(\sum\limits^n_{k=1}\,d_{ik}\,x_k\right)^2=\left(\sum\limits^n_{k=1}\,x_k\left(d_{ik}-\frac{1}{n}\,\sum\limits^n_{k=1}\,d_{ik}\right)\right)^2\leq \sum\limits^n_{k=1}\,d^2_{ik}-\frac{1}{n}\,\left(\sum\limits^n_{k=1}\,d_{ik}\right)^2\,.$$\vspace*{3mm}
Taking summation from $i=1$ to $n$, we get
$$\sum\limits^n_{i=1}\,\left(\sum\limits^n_{k=1}\,d_{ik}\,x_k\right)^2\leq 2\,\sum\limits_{1\leq i<j\leq n}\,d^2_{ij}-\frac{1}{n}\,\sum\limits^n_{i=1}\,D^2_i\,.$$\vspace*{3mm}
Using (\ref{s2}) with the above result, we have
\begin{equation}
(\partial^L_1(G)-D_1)^2\leq \sum\limits^n_{i=1}\,(\partial^L_1(G)-D_i)^2\,x^2_i\leq 2\,\sum\limits_{1\leq i<j\leq n}\,d^2_{ij}-\frac{1}{n}\,\sum\limits^n_{i=1}\,D^2_i\,.\label{s4}
\end{equation}

\vspace*{1mm}

Now we only need to show that the inequality in (\ref{s1}) is strict. Assume to the contrary that the equality holds in (\ref{s1}).
Then all the inequalities in the above must be equalities. From the equality in (\ref{s4}), we get $D_1=D_2=\cdots=D_n=D, \text{(say)}$. From the equality in (\ref{s3}), we have
$$\frac{d_{i1}-\frac{1}{n}\,D_i}{x_1}=\frac{d_{i2}-\frac{1}{n}\,D_i}{x_2}=\cdots=\frac{d_{in}-\frac{1}{n}\,D_i}{x_n}~~~\mbox{ for each $i\in [n]$, }$$
that is,
$$\frac{d_{i1}-\frac{1}{n}\,D}{x_1}=\frac{d_{i2}-\frac{1}{n}\,D}{x_2}=\cdots=\frac{d_{in}-\frac{1}{n}\,D}{x_n}=r_i,~\mbox{ (say), }~~~\mbox{ for each $i\in [n]$.}$$
If $r_i=0$, then $d_{i1}=d_{i2}=\cdots=d_{in}=\frac{1}{n}\,D$, a contradiction. Otherwise, $r_i\neq 0$ for any $i$.
Since $d_{ii}=0$, we have $r_i x_i=-\frac{1}{n}D$ for $i\in [n]$. Moreover, for $i\neq j$, $d_{ij}=r_i x_j+\frac{1}{n}D=\frac{D}{n}\left(1-\frac{r_i}{r_j}\right)$ and
$d_{ji}=\frac{D}{n}\left(1-\frac{r_j}{r_i}\right)$. Since $d_{ij}=d_{ji}$, we must have
$r_i=r_j$. Hence $d_{ij}=0$, a contradiction. This completes the proof.\hspace*{\fill}$\Box$
\end{proof}

\section{Distance Laplacian spectral radius and clique number}
\hspace*{\parindent}In this section, we focus on the distance Laplacian spectral radii of connected graphs with a given clique number.
We consider the following question.

\emph{Which graphs attain the maximum and minimum distance Laplacian spectral radii among all connected graphs with a given clique number?}

In this section, we give the following two main results.

\begin{theorem}\label{th15}
Let $G$ be a connected graph with order $n$ and clique number $\omega$. Then
$$\partial^L_1(G)\geq \partial^L_1(T_{n,\omega})=n+\Big\lceil \frac{n}{\omega}\Big\rceil.$$
Moreover, if $n=k\omega$ or $n=k\omega-1$, then equality holds if
and only if $G\cong T_{n,\omega}$. If
$n=(k-1)\omega+t$~$(0<t<\omega-1)$, then equality holds if and only
if $G\cong H_{1}\vee\cdots\vee H_{t}\vee H$, where $H_{i}~(1\leq
i\leq t)$ is an empty graph of order $k$ and $H$ is a graph of order
$n-kt$ with clique number $\omega-t$ and algebraic connectivity not
less that $n-k(t+1)$.
\end{theorem}

\begin{theorem}\label{th19}
Let $G$ be a connected graph with order $n$ and clique number $\omega$. Then
$\partial_{1}^{L}(G)\leq \partial_{1}^{L}(K_{\omega}^{n-\omega})$, and the equality holds if and only if $G\cong K_{\omega}^{n-\omega}$.
\end{theorem}

Recently, the Theorem \ref{th19} was proved by Lin and Zhou in \cite{LZ2}. Here, we give a different proof.

First, we consider the graphs attain the minimum distance Laplacian spectral radii among all connected graphs with a given clique number.

Let $K_{n_{1},\ldots,n_{k}}$ be a complete $k$-partite graph. All the distance Laplacian eigenvalues of $K_{n_{1},\ldots,n_{k}}$ is given by the following result.

\begin{lemma}\label{le11}
Let $G=K_{n_{1},\ldots,n_{k}}$ be a complete k-partite graph with order $n$. Then the distance Laplacian characteristic
polynomial of $G$ is
$$P(x)=x(x-n)^{k-1}\prod_{i=1}^{k}(x-n-n_{i})^{n_{i}-1}.$$
\end{lemma}
\begin{proof}
Clearly, the distance Laplacian matrix of $G$ is
\begin{equation*}\begin{split}
D^{L}(G)&=\left(\begin{array}{ccccccc}
(n+n_{1})I_{n_{1}}\!\!-2J_{n_{1}\times n_{1}} & -J_{n_{1}\times n_{2}} & \cdots & -J_{n_{1}\times n_{k}}\\
-J_{n_{2}\times n_{1}} & (n+n_{2})I_{n_{2}}\!\!-2J_{n_{2}\times n_{2}} & \cdots & -J_{n_{2}\times n_{k}}\\
\vdots & \vdots & \ddots & \vdots\\
-J_{n_{k}\times n_{1}} & -J_{n_{k}\times n_{2}} &   \cdots & (n+n_{k})I_{n_{k}}\!\!-2J_{n_{k}\times n_{k}}
\end{array}\right).
\end{split}\end{equation*}
Then
\begin{equation*}\begin{split}
\mathrm{det}(x I-D^{L}(G))&=\prod_{i=1}^{k}(x-n-n_{i})^{n_{i}-1}\left|\begin{array}{ccccccc}
x-n+n_{i} & n_{2} & \cdots & n_{k}\\
n_{1} & x-n+n_{2} & \cdots & n_{k}\\
\vdots & \vdots & \ddots & \vdots\\
n_{1} & n_{2} &   \cdots & x-n+n_{k}
\end{array}\right|\\
&=x(x-n)^{k-1}\prod_{i=1}^{k}(x-n-n_{i})^{n_{i}-1},
\end{split}\end{equation*}
as required.\hspace*{\fill}$\Box$
\end{proof}

Using Lemma \ref{le11}, we get the distance Laplacian spectral radius of Tur\'{a}n graph.

\begin{corollary}\label{co12}
Let $T_{n,\omega}$ be the $\omega$-partite Tur\'{a}n graph on $n$ vertex. Then $\partial^L_1(T_{n,\omega})=n+\lceil \frac{n}{\omega}\rceil$.
\end{corollary}

The following result is the well-known Tur\'{a}n theorem (see, for example, in \cite{BB}).

\begin{lemma}\label{le13}
Let $G$ be a graph with $n$ vertices without an $(\omega+1)$-clique. Then $|E(G)|\leq |E(T_{n,\omega})|$ and the equality holds if and only if $G\cong T_{n,\omega}$.
\end{lemma}

Let $G$ be a graph with clique number $\omega$. According to Lemma \ref{le13}, it is easy to see that the minimum degree $\delta(G)\leq n-\lceil\frac{n}{\omega}\rceil$.

Let $Diag(G)$ be the diagonal matrix of vertex degrees of $G$ and $A(G)$ be the $(0,1)$ adjacency matrix of $G$. The matrix $L(G)=Diag(G)-A(G)$ is called the Laplacian matrix of $G$. Let $\sigma_{1}(G)\geq \sigma_{2}(G)\geq \cdots\geq \sigma_{n}(G)=0$ denote the eigenvalues of $L(G)$. It is well know that $\sigma_{n-1}(G)>0$ if and only if $G$ is connected. We call $\alpha(G)=\sigma_{n-1}(G)$ the \emph{algebraic connectivity} of $G$. If the diameter is 2, then the distance Laplacian spectral radius will be given by the algebraic connectivity.

\begin{lemma}{\bf(\cite{AM})}\label{le14-}
Let $G$ be a connected graph on $n$ vertices with diameter $\mathrm{diam}(G)\leq 2$. Then $\partial^L_1(G)=2n-\alpha(G)$.
\end{lemma}

\begin{lemma}{\bf(\cite{JY})}\label{le14}
Let $G$ be a non-complete graph on $n$ vertices with clique number $\omega(G)\leq \omega$. Then
$$\alpha(G)\leq \alpha(T_{n,\omega}).$$
Moreover, if $n=k\omega$ or $n=k\omega-1$, then equality holds if and only if $G\cong T_{n,\omega}$. If $n=(k-1)\omega+t~(0<t<\omega-1)$, then equality holds if and only if there exist graphs $H_{1},\ldots,H_{t}$ of order $k$ with no edges and $H$ of order $n-kt$ not containing $K_{\omega+1-t}$ such that $G=H_{1}\vee\ldots\vee H_{t}\vee H$ and $\alpha(H)\geq n-k(t+1)$.
\end{lemma}

Now, we are ready to present the proof of Theorem 5.1.

\noindent{\bf Proof of Theorem 5.1.}
Let $k=\lceil\frac{n}{\omega}\rceil$. If $\mathrm{diam}(G)=2$, then it follows from Lemma \ref{le14-} that $\partial^L_1(G)=2n-\alpha(G)$. Hence, the Theorem follows immediately from Lemma \ref{le14}. So in the following, we may assume $\mathrm{diam}(G)\geq 3$. Let $v\in V(G)$ with $d(v)=\delta(G)$. Since $\delta(G)\leq n-k$ (due to Tur\'{a}n theorem), it follows that
\begin{eqnarray*}
D_{v}&\geq& \delta(G)+2(n-1-\delta(G))\\
&=&2n-2-\delta(G)\\
&\geq&n+k-2.
\end{eqnarray*}
According to Theorem \ref{th9}, we have $\partial^L_1(G)>D_{v}+2\geq n+k=\partial^L_1(T_{n,\omega})$. This completes the proof.\hspace*{\fill}$\Box$

In order to prove Theorem 5.2, we give some graft transformations on the distance Laplacian spectral radius. These two graft transformations in Theorem \ref{th16} and Theorem \ref{th18} were also given in \cite{LZ2}. We prove them by using a different method in this paper.

Let $G$ be a connected graph and $X$ be a unit eigenvector of $G$ corresponding to $\partial^L_1(G)$. It will be convenient to associate with $X$ a labelling of $G$ in which vertex $v_{i}$ is labeled as $x_{v_{i}}$. It is easy to see that the vector $\mathds{1}=(1,1,\ldots,1)^{t}$ is an eigenvector of $\mathcal{L}(G)$ corresponding to $\partial^L_n(G)=0$. Then we have $X\bot \mathds{1}$ and $\sum\limits_{v_{i}\in V(G)}x_{v_{i}}=0$. Recall that
$$\partial^L_1(G)=X^{t}\mathcal{L}(G)X=\sum_{\{v_{i},v_{j}\}\in V(G)}d(v_{i},v_{j})(x_{v_{i}}-x_{v_{j}})^2.$$
The eigenvalue equation $\mathcal{L}(G)X=\partial^L_1(G)X$ should be interpreted as
$$\partial^L_1(G)x_{v_{i}}=\sum_{v_{j}\in V(G)}d(v_{i},v_{j})(x_{v_{i}}-x_{v_{j}})=D_{v_{i}}x_{v_{i}}-\sum_{v_{j}\in V(G)\backslash\{v_{i}\}}d(v_{i},v_{j})x_{v_{j}}.$$

\begin{center}
\scalebox{0.7}[0.7]
{\includegraphics{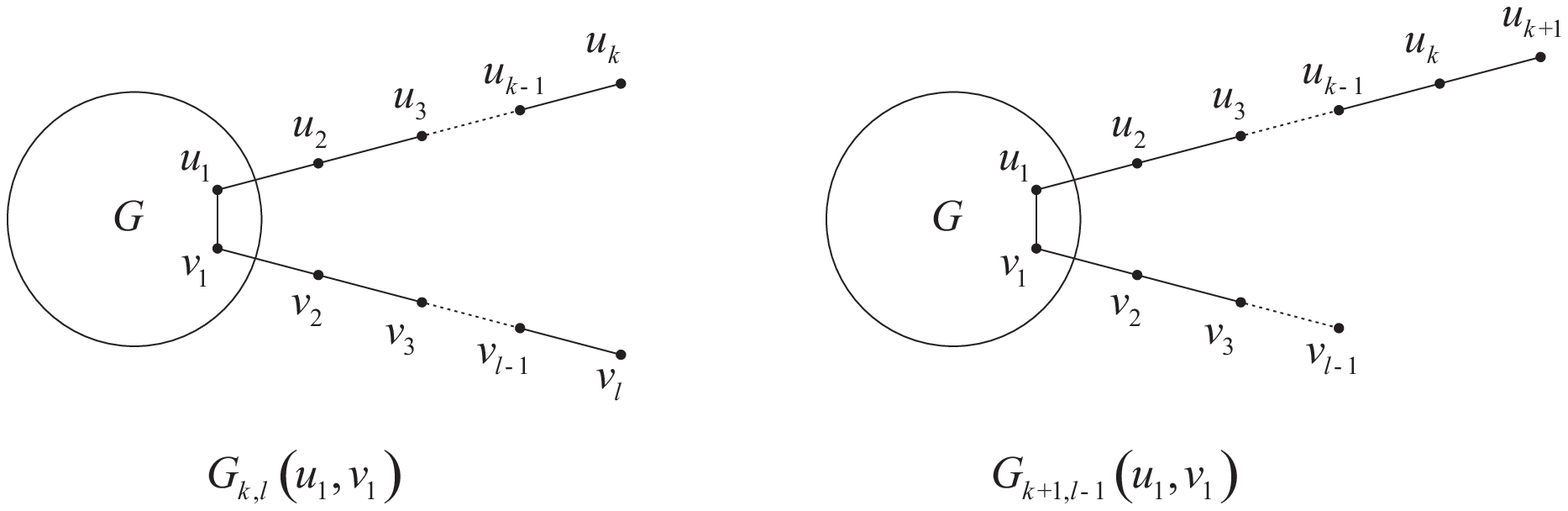}}
\vskip 0.3cm Fig.$1$. Graphs $G_{k,l}(u_{1},v_{1})$ and $G_{k+1,l-1}(u_{1},v_{1})$.
\end{center}

Let $G$ be a connected graph with vertex set $V(G)=\{u_{1},v_{1},w_{1},w_{2},\ldots ,w_{n-k-l}\}$, where $n>k+l$. Suppose that vertex $u_{1}$ is adjacent to vertex $v_{1}$ and $N_{G}(u_{1})=N_{G}(v_{1})$. Adding an edge between $u_{1}$ and an end-vertex of a path $P_{k-1}$, then performing the same operation between $v_{1}$ and an end-vertex of a path $P_{l-1}$, one obtains a connected graph, denote by $G_{k,l}(u_{1},v_{1})$.

\begin{theorem}\label{th16}
Let $G_{k,l}(u_{1},v_{1})$ be the graph defined above. If $k\geq l\geq 2$, then
$$\partial^L_1(G_{k+1,l-1}(u_{1},v_{1}))\geq \partial^L_1(G_{k,l}(u_{1},v_{1})).$$
Moreover, if the equality holds, then there exists a unit eigenvector corresponding to $\partial^L_1(G_{k,l})$ taking the same value on vertices of $V(G)$.
\end{theorem}
\begin{proof}
For simplicity, we may take $\partial=\partial^L_1(G_{k,l}(u_{1},v_{1}))$ and $\partial^{\prime}=\partial^L_1(G_{k+1,l-1}(u_{1},v_{1}))$. Let $$X=(x_{w_{1}},\ldots,x_{w_{n-k-l}},x_{u_{1}},\ldots,x_{u_{k}},x_{v_{1}},\ldots,x_{v_{l}})^t$$ be a unit eigenvector corresponding to $\partial$. Constructing a new vector $$Z=(z_{w_{1}},\ldots,z_{w_{n-k-l}},z_{u_{1}},\ldots,z_{u_{k+1}},z_{v_{1}},\ldots,z_{v_{l-1}})^t$$ on the vertices of $G_{k+1,l-1}(u_{1},v_{1})$ such that
\[
\left\{
\begin{array}{l}
z_{u_{1}}=x_{v_{1}},\\
z_{u_{i}}=x_{u_{i-1}}\hspace*{12pt}\text{for}\hspace*{12pt}2\leq i\leq k+1,\\
z_{v_{i}}=x_{v_{i+1}}\hspace*{13pt}\text{for}\hspace*{12pt}1\leq i\leq l-1,\\
z_{w_{i}}=x_{w_{i}}\hspace*{18pt}\text{for}\hspace*{12pt}1\leq i\leq n-k-l.
\end{array}
\right.
\]
Obviously, $Z$ is also a unit vector. Hence
\begin{eqnarray*}
\partial^{\prime}-\partial&\geq& Z^{t}\mathcal{L}(G_{k+1,l-1}(u_{1},v_{1}))Z-X^{t}\mathcal{L}(G_{k,l}(u_{1},v_{1}))X\\
&=&\sum_{j=1}^{k}\sum_{i=1}^{n-k-l}(x_{u_{j}}-x_{w_{i}})^{2}
-\sum_{j=2}^{l}\sum_{i=1}^{n-k-l}(x_{v_{j}}-x_{w_{i}})^{2}.
\end{eqnarray*}
Let $$\mathbf{\Sigma}=\sum_{j=1}^{k}\sum_{i=1}^{n-k-l}(x_{u_{j}}-x_{w_{i}})^{2}
-\sum_{j=2}^{l}\sum_{i=1}^{n-k-l}(x_{v_{j}}-x_{w_{i}})^{2}.$$ So in the following, we only need to show that $\mathbf{\Sigma}\geq 0$.

\noindent{\bf Case 1.} $k\geq l+1$.

Let $X^{*}$ be a vector on the vertices of $G_{k,l}(u_{1},v_{1})$,
such that
\[
\left\{
\begin{array}{l}
x^{*}_{w_{i}}=x_{w_{i}}\hspace*{38pt}\text{for}\hspace*{12pt}1\leq i\leq n-k-l,\\
x^{*}_{v_{i}}=x_{u_{k-l+i}}\hspace*{23pt}\text{for}\hspace*{12pt}1\leq i\leq l,\\
x^{*}_{u_{i}}=x_{u_{k-l-i+1}}\hspace*{14pt}\text{for}\hspace*{12pt}1\leq i\leq k-l,\\
x^{*}_{u_{i}}=x_{v_{i-k+l}}\hspace*{23pt}\text{for}\hspace*{12pt}k-l+1\leq i\leq k.
\end{array}
\right.
\]
It is clear that $X^{*}$ is also a unit vector, then we obtain
\begin{eqnarray*}
&&X^{t}\mathcal{L}(G_{k,l}(u_{1},v_{1}))X-{X^{*}}^{t} \mathcal{L}(G_{k,l}(u_{1},v_{1}))X^{*}\\
&=&\sum_{j=k-l+1}^{k}\sum_{i=1}^{n-k-l}(k-l)(x_{u_{j}}-x_{w_{i}})^{2}
+\sum_{j=1}^{k-l}\sum_{i=1}^{n-k-l}(2j-k+l-1)(x_{u_{j}}-x_{w_{i}})^{2}\\
&&~~~~~~~~~~~~~~~~~~~~~~~~~~~~~~~-\sum_{j=1}^{l}\sum_{i=1}^{n-k-l}(k-l)(x_{v_{j}}-x_{w_{i}})\\
&\geq&0.
\end{eqnarray*}
Since $k-l> 2j-k+l-1$ for $1\leq j\leq k-l$, it follows that
\begin{eqnarray*}
&&(k-l)\mathbf{\Sigma}\\
&=&\left(k-l\right)\left(\sum_{j=1}^{k}\sum_{i=1}^{n-k-l}(x_{u_{j}}-x_{w_{i}})^{2}
-\sum_{j=2}^{l}\sum_{i=1}^{n-k-l}(x_{v_{j}}-x_{w_{i}})^{2}\right)\\
&\geq&(k-l)\left(\sum_{j=1}^{k}\sum_{i=1}^{n-k-l}(x_{u_{j}}-x_{w_{i}})^{2}
-\sum_{j=1}^{l}\sum_{i=1}^{n-k-l}(x_{v_{j}}-x_{w_{i}})^{2}\right)\\
&\geq&\sum_{j=k-l+1}^{k}\sum_{i=1}^{n-k-l}(k-l)(x_{u_{j}}-x_{w_{i}})^{2}
+\sum_{j=1}^{k-l}\sum_{i=1}^{n-k-l}(2j-k+l-1)(x_{u_{j}}-x_{w_{i}})^{2}\\
&&-\sum_{j=1}^{l}\sum_{i=1}^{n-k-l}(k-l)(x_{v_{j}}-x_{w_{i}})\\
&\geq&0.
\end{eqnarray*}
This implies $\mathbf{\Sigma}\geq 0$, therefore $\partial^{\prime}\geq \partial$. Further, if $\partial^{\prime}=\partial$, then $(k-l)\mathbf{\Sigma}=0$. According to the above inequality, it can easily be seen that $\sum\limits_{i=1}^{n-k-l}(x_{u_{1}}-x_{w_{i}})^{2}=0$ and $\sum\limits_{i=1}^{n-k-l}(x_{v_{1}}-x_{w_{i}})^{2}=0$. Consequently, $x_{u_{1}}=x_{v_{1}}=x_{w_{i}}$ for $1\leq i\leq n-k-l$. Hence the result holds.

\noindent{\bf Case 2.} $k=l\geq 1$.

Let $Y=(y_{w_{1}},\ldots,y_{w_{n-k-l}},y_{u_{1}},\ldots,y_{u_{k}},y_{v_{1}},\ldots,y_{v_{k}})^t$ be a unit eigenvector corresponding to $\partial$. Giving a vector
$Y^{'}=(y_{w_{1}},\ldots,y_{w_{n-k-l}},y_{v_{1}},\ldots,y_{v_{k}},y_{u_{1}},\ldots,y_{u_{k}})^t$, which is obtained from $Y$ by exchanging the values on $v_{i}$ and $u_{i}$ for $1\leq i\leq k$. It is clear that $Y^{'}$ is also a unit eigenvector corresponding to $\partial$. Let $Y^{''}=Y+Y^{'}$. We divide our proof into two subcases.

\noindent{\bf Subcase 2.1.} $Y^{''}\neq \mathbf{0}$.

If $Y^{''}\neq \mathbf{0}$, then we take $X=\alpha Y^{''}$ where $\alpha$ is a constant such that $X$ is a unit vector. Clearly, $X$ is a unit eigenvector corresponding to $\mu$ and $x_{u_{i}}=x_{v_{i}}$ for $1\leq i\leq k$. Then we have $\mathbf{\Sigma}=\sum\limits_{i=1}^{n-2k}(x_{u_{1}}-x_{w_{i}})^{2}\geq 0$, and so $\partial^{\prime}\geq \partial$. If equality holds, then $\mathbf{\Sigma}=\sum\limits_{i=1}^{n-2k}(x_{u_{1}}-x_{w_{i}})^{2}=0$, so $x_{w_{i}}=x_{u_{1}}$. Hence the result holds.

\noindent{\bf Subcase 2.2.} $Y^{''}=\mathbf{0}$.

If $Y^{''}=\mathbf{0}$, then we have $Y=(0,\ldots,0,y_{u_{1}},\ldots,y_{u_{k}},-y_{u_{1}},\ldots,-y_{u_{k}})^t$. Let $X=\alpha Y$ be a unit vector where $\alpha$ is a constant. Thus $X$ is a unit eigenvector corresponding to $\mu$. Note that $x_{w_{i}}=0$ for $1\leq i\leq n-2k$ and $x_{v_{i}}=-x_{u_{i}}$ for $1\leq i\leq k$ in this case. Therefore $\mathbf{\Sigma}=(n-2k)x_{u_{1}}^{2}\geq 0$ and equality holds if and only if $x_{u_{1}}=0$. It is evident to see that the theorem holds. This completes the proof. \hspace*{\fill}$\Box$
\end{proof}

\begin{corollary}\label{co17}
Let $G_{k,2}(u_{1},v_{1})$ be the graph defined above. If $k\geq 2$, then $$\partial_{1}^{L}(G_{k+1,1}(u_{1},v_{1}))> \partial_{1}^{L}(G_{k,2}(u_{1},v_{1})).$$
\end{corollary}
\begin{proof}
We prove this corollary by contradiction. Suppose that $\partial_{1}^{L}(G_{k+1,1}(u_{1},v_{1}))=\partial_{1}^{L}(G_{k,2}(u_{1},v_{1})).$ We may take $\partial=\partial_{1}^{L}(G_{k,2}(u_{1},v_{1}))$. It follows from Theorem \ref{th16} that there exists a unit eigenvector $X$ corresponding to $\partial$ such that $x_{w_{i}}=x_{u_{1}}=x_{v_{1}}$ for $1\leq i\leq n-k-2$. By eigenvalue equation, we have
\begin{equation}
\partial x_{u_{1}}=D_{u_{1}}x_{u_{1}}-\sum_{i=1}^{n-k-2}d(u_{1},w_{i})x_{w_{i}}-x_{v_{1}}-2x_{v_{2}}-\sum_{i=2}^{k}(i-1)x_{u_{i}} \label{5.1}
\end{equation}
and
\begin{equation}
\partial x_{v_{1}}=D_{v_{1}}x_{v_{1}}-\sum_{i=1}^{n-k-2}d(v_{1},w_{i})x_{w_{i}}-x_{u_{1}}-x_{v_{2}}-\sum_{i=2}^{k}ix_{u_{i}}. \label{5.2}
\end{equation}
Note that $D_{v_{1}}=D_{u_{1}}+k-2$ and $X\bot \mathds{1}$. Taking the difference between equations (1) and (2), we have
\begin{equation}
(n-2)x_{v_{1}}+2x_{v_{2}}=0.\label{5.3}
\end{equation}
Similarly
\begin{eqnarray*}
\partial x_{v_{1}}-\partial x_{v_{2}}&=&D_{v_{1}}x_{v_{1}}-D_{v_{2}}x_{v_{2}}+x_{v_{1}}-x_{v_{2}}+
\sum_{i=1}^{n-k-2}x_{w_{i}}+\sum_{i=1}^{k}x_{u_{i}}\\
&=&D_{v_{1}}x_{v_{1}}-(D_{v_{2}}+2)x_{v_{2}}.
\end{eqnarray*}
Then
\begin{equation}
(\partial-D_{v_{1}})x_{v_{1}}=(\partial-D_{v_{2}}-2)x_{v_{2}}.\label{5.4}
\end{equation}
Since $\partial>D_{v_{2}}+2$ (by Theorem \ref{th9}), we have
$\left\{\begin{array}{l}
x_{v_{1}}x_{v_{2}}\leq 0\\
x_{v_{1}}x_{v_{2}}\geq 0
\end{array}\right.
$ from equations (\ref{5.3}) and (\ref{5.4}). Therefore, $x_{v_{1}}=x_{v_{2}}=0.$ These two identities $\sum\limits_{i=2}^{k}x_{u_{i}}=0$ and $\sum\limits_{i=2}^{k}(i-1)x_{u_{i}}=0$ will be implied through equations (\ref{5.1}) and (\ref{5.2}). It follows that
$$\partial x_{u_{2}}=D_{u_{2}}x_{u_{2}}-\sum_{i=3}^{k}(i-2)x_{u_{i}}=D_{u_{2}}x_{u_{2}}.$$
The relation $\partial>D_{u_{2}}$ can be obtained from Theorem \ref{th9}, further we figure out $x_{u_{2}}=0$. Analogously, these parameters admit $x_{u_{3}}=\cdots =x_{u_{k}}=0$. This implies $X=\mathbf{0}$, a contradiction. Thus we complete the proof.\hspace*{\fill}$\Box$
\end{proof}

Let $u$ be a vertex of a connected graph $G$. Denote by $G_{k,l}(u)$ the graph obtained from $G\cup P_{k-1}\cup P_{l-1}$ by adding two edges between $u$ and end-vertices of $P_{k-1}$ and $P_{l-1}$. An argument similar to the one used in Theorem \ref{th16} and Corollary \ref{co17}, we obtain the following result.

\begin{theorem}\label{th18}
Let $G_{k,l}(u)$ be the graph defined above. If $k\geq l\geq 2$, then $\partial_{1}^{L}(G_{k+1,l-1}(u))\geq \partial_{1}^{L}(G_{k,l}(u)).$
Moreover, if $l=2$, then $\partial_{1}^{L}(G_{k+1,1}(u))>\partial_{1}^{L}(G_{k,2}(u)).$
\end{theorem}

Using Theorem \ref{th18}, we immediately get the path $P_n$ attains the maximum distance spectral radius among all graphs which is a conjecture from Aouchiche and Hensen \cite{AM1,AH},
was solved by da Silva Jr. and  Nikiforov \cite{SJ}.

\vspace{4mm}

\begin{center}
\scalebox{0.6}[0.6]
{\includegraphics{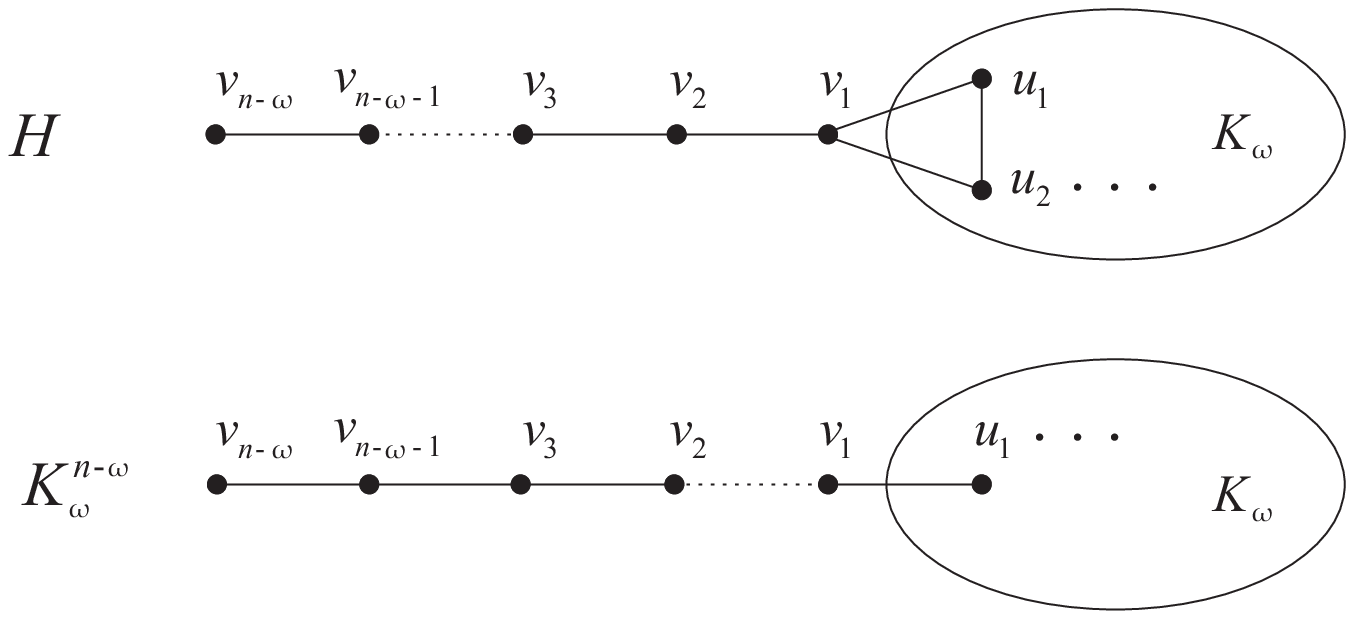}}
\vskip 0.3cm Fig.$2$. Graphs $K_{\omega}^{n-\omega}$ and $H$.
\end{center}

Now, we are ready to give the proof of the Theorem 5.2.

\noindent{\bf Proof of Theorem 5.2.}
Recall that $\partial_{1}^{L}(G)\leq \partial_{1}^{L}(P_{n})$ with equality holds if and only if $G\cong P_{n}$(see \cite{SJ}). According to this conclusion, it is evident to see that the theorem holds when $\omega=2$. So in the following, we may assume that $\omega\geq 3$. Suppose that $G$ is an extremal graph which attains the maximum distance Laplacian spectral radius. It is suffices to show that $G\cong K_{\omega}^{n-\omega}$. Let $S$ be an $\omega$-clique of $G$.

\noindent{\bf Claim 1.} There is exactly one component in $G-S$.

Otherwise, suppose that there are $t~(\geq 2)$ components $G_{1},G_{2},\ldots, G_{t}$ in $G-S$. For $i=1$ and 2, by deleting edges in $E_{G}[S,V(G_{i})]$ (if necessary), we get a graph $G'$ from $G$ such that $E_{G'}[S,V(G_{i})]$ is a cut edge in $G'$. With the monotonicity of distance Laplacian spectral radius when removing edges, it admits $\partial_{1}^{L}(G')\geq\partial_{1}^{L}(G)$. Let $E_{G'}[S,V(G_{1})]=e_{1}$ and $E_{G'}[S,V(G_{2})]=e_{2}$. Suppose that $e_{1}$ and $e_{2}$ are incident with a same vertex in $V(S)$. Combining the monotonicity of distance Laplacian spectral radius and Theorem \ref{th18}, it is easy to verify that there is a graph $G''$ such that $\partial_{1}^{L}(G'')>\partial_{1}^{L}(G')$, a contradiction. Suppose that $e_{1}$ and $e_{2}$ are incident with two different vertices. Analogously, we can also get a contradiction in accordance with Theorem \ref{th16} and Corollary \ref{co17}. Thus claim holds.

\noindent{\bf Claim 2.} There is exactly one vertex of $V(G)\backslash S$, which is adjacent to vertices in $S$.

Suppose that there are two vertices $v_{1},v_{2}\in V(G)\backslash S$ adjacent to vertices in $S$. Let $E^{*}$ be a minimal edge cut separating $v_{1}$ and $v_{2}$ in $G-S$, and let $G'=G-E^{*}$. Then $\partial_{1}^{L}(G')\geq \partial_{1}^{L}(G)$. However, there are two components in $G'-S$, this contradicts claim 1. Thus claim holds.

\noindent{\bf Claim 3.} $G-S$ is a path.

We first assume that there is a vertex $v\in V(G)\backslash S$ with $d_{G-S}(v)\geq 3$. Let $G'$ be a graph obtained from $G$ by deleting edges in $G-S$ such that $G'-S$ is a tree and the degree of $v$ is unchanged. There are at least three components in $G'-S-v$. Then by an argument similar to the proof of claim 1, we can find a graph $G''$ such that $\partial_{1}^{L}(G'')>\partial_{1}^{L}(G')\geq \partial_{1}^{L}(G)$, a contradiction. Thus we infer that $G-S$ is a cycle or a path. Suppose that $G-S$ is a cycle. Denote by $G-S=v_{1}v_{2}\ldots v_{n-\omega}v_{1}$. Let $G'''$ be a graph obtained from $G$ by deleting some edges of $E_{G}[S,V(G)\backslash S]$ such that $E_{G'''}[S,V(G)\backslash S]$ is a cut edge of $G'''$. Without loss of generality, we may assume $v_{1}$ is an end-vertex of this cut edge. According to Theorem \ref{th18}, we have $\partial_{1}^{L}(K_{\omega}^{n-\omega})>\partial_{1}^{L}(G'''-v_{2}v_{3})\geq \partial_{1}^{L}(G)$, a contradiction. Therefore the claim holds.

Let $G-S=v_{1}v_{2}\ldots v_{n-\omega}$ be a path. By claim 2, there is exactly one vertex $v\in \{v_{1},v_{2},\ldots, v_{n-\omega}\}$ that adjacent to vertices in $S$. If $v\notin \{v_{1},v_{n-\omega}\}$, then it follows from Theorem \ref{th18} that $\partial_{1}^{L}(K_{\omega}^{n-\omega})>\partial_{1}^{L}(G)$, a contradiction. So, we may assume that $v=v_{1}$. If $v_{1}$ is adjacent to exactly one vertex in $S$, then $G\cong K_{\omega}^{n-\omega}$. Otherwise, it is easy to verify that $\partial_{1}^{L}(G)\leq \partial_{1}^{L}(H)$, where $H$ is the graph shown in Fig. 2. Let $X=(x_{u_{1}},x_{u_{2}},\ldots,x_{u_{\omega}},x_{v_{1}},x_{v_{2}},\ldots,x_{v_{n-\omega}})$ be a unit eigenvector corresponding to $\partial_{1}^{L}(H)$. Then $$\partial_{1}^{L}(K_{\omega}^{n-\omega})-\partial_{1}^{L}(H)\geq X^{t}\mathcal{L}(K_{\omega}^{n-\omega})X-X^{t}\mathcal{L}(H)X=\sum_{i=1}^{n-\omega}(x_{u_{2}}-x_{v_{i}})^{2}\geq 0.$$
If $\sum\limits_{i=1}^{n-\omega}(x_{u_{2}}-x_{v_{i}})^{2}=0$, we obtain $x_{u_{2}}=x_{v_{1}}=x_{v_{2}}=\cdots=x_{v_{n-\omega}}$. It is easy to see that $x_{u_{1}}=x_{u_{2}}$ and $x_{u_{i}}=x_{u_{3}}$ for $3\leq i\leq \omega$. By eigenvalue equation, we have
$$\partial_{1}^{L}(H)x_{v_{1}}=(2\omega-2)x_{v_{1}}-x_{u_{1}}-x_{u_{2}}-2(\omega-2)x_{u_{3}}$$
and
$$\partial_{1}^{L}(H)x_{u_{1}}=(\omega-1)x_{v_{1}}-x_{u_{2}}-(\omega-2)x_{u_{3}}.$$
Hence $(\omega-2)x_{u_{1}}=(\omega-2)x_{u_{3}}$, that is, $x_{u_{1}}=x_{u_{3}}$. Therefore $X=\alpha\mathds{1}$, where $\alpha$ is a constant. This leads to a contradiction. Consequently, the inequality $\partial_{1}^{L}(K_{\omega}^{n-\omega})>\partial_{1}^{L}(H)$ holds. This is contrary to the maximality of $\partial_{1}^{L}(G)$. Thus we complete the proof.\hspace*{\fill}$\Box$

\section{The distance signless Laplacian eigenvalues of graphs}
\hspace*{\parindent}In this section, we will give some results about
distance signless Laplacian eigenvalues. Let $G$ be a connected
graph with vertices set $\{v_{1},v_{2},\ldots,v_{n}\}$. Let
$X=(x_{1},x_{2},\ldots,v_{n})^{t}$ be a unit eigenvector of $G$
corresponding to $\partial_1^Q(G)$. Recall that
$$\partial_1^Q(G)=X^{t}\mathcal{Q}(G)X=\sum_{\{v_{i},v_{j}\}\in V(G)}d(v_{i},v_{j})(x_{i}+x_{j})^2.$$ Let $G+e$ be a graph obtained from $G$ by adding an edge. Note
that the distance between each pair of vertices is nonincreasing and one pair is decreasing. It is straightforward to show that $$\partial_1^Q(G)>\partial_1^Q(G+e).$$
Using this inequality, we have $\partial_1^Q(G)\geq\partial_1^Q(K_{n})=2n-2$. However, most of the time, if $G$ is not a complete graph, then the distance signless
Laplacian spectral radius of $G$ is  considerably larger than $2n-2$.  Then we establish a lower bound for $\partial_1^Q(G)$.

\begin{theorem}\label{l2}
Let $G$ be a connected graph with order $n$ and diameter $d\geq 3$. Then
$$\partial_1^Q(G)>2n-4+2d.$$
Further, if the diameter is at least $4$, then we have $\partial_1^Q(G)>\frac{d(n+2)}{2}$.
\end{theorem}
\begin{proof}
Let $u,\,w$ be the end-vertices of a diametral path of $G$. Let $X$
be a vector, where the values of the components corresponding to
$u,w$ are 1, otherwise the value is 0. If $d(u,w)\geq 4$, then we
have
$$\partial_1^Q(G)>\frac{X^{t}\mathcal{Q}(G)X}{X^{t}X}=\frac{D_{u}+D_{w}+2d}{2}\geq \frac{(n-2)d+4d}{2}\geq 2n-4+2d.$$
If $d(u,w)=3$, then we can decompose $V(G)=V_{0}\cup V_{1}\cup V_{2}\cup V_{3}$, where $V_{0}=\{u\}$, $V_{1}=\{v\in V(G): d(u,v)=1\}$, $V_{2}=\{v\in V(G): d(u,v)\geq 2\}\backslash \{w\}$ and $V_{3}=\{w\}$. We may assume that $|V_{1}|=n_{1}$ and $|V_{2}|=n_{2}$. Note that $\partial_1^Q(G)>\partial_1^Q(G+e)$ where $e\notin E(G)$. Then we can add as many edges as possible until $G[V_{i}\cup V_{i+1}]$ is a clique for $i=0,1,2$. We denote the new graph by $G'$. Consider the quotient matrix $R$ of $D^{Q}(G')$ for the partition of the vertex set $V(G')$ into 4 parts $\{u\}$, $V_{1}$, $V_{2}$ and $\{w\}$. This quotient matrix can be expressed as
\[R=\left(\begin{array}{cccccccc}
n_1+2n_2+3 & n_1 & 2n_2 & 3\\
1 & 2n_1+n_2+1 & n_2 & 2 \\
2 & n_1 & n_1+2n_2+1 & 1\\
3 & 2n_1 & n_2 & 2n_1+n_2+3
\end{array}\right).
\]
By a simple calculation, we have
$$\mathrm{det}((2n-4+2d)I-R)=-4(n_{1}^{3}+8n_{1}^{2}+15n_{1}+n_{2}^{3}+8n_{2}^{2}+15n_{2})<0,$$
this implies $\lambda_{1}(R)>2n-4+2d$. Consequently, we have $\partial_1^Q(G)\geq \partial_1^Q(G')\geq \lambda_{1}(R)>2n-4+2d$, as required.\hspace*{\fill}$\Box$
\end{proof}

Let $G$ be a connected graph of order $n$. It is easy to see that $\partial^Q_1(G)\geq\partial^L_1(G)$. Further, if $n\geq 3$, then $\partial^Q_1(G)>\partial^L_1(G)$. However, a question arises: how large the value of $\partial^Q_1(G)-\partial^L_1(G)$ can be? We give an upper bound for $ \partial^Q_1(G)-\partial^L_1(G)$ when the diameter is small. In order to prove the result, we need the following lemma.

 \begin{lemma} {\bf(\cite{AM})} \label{dv1} Let $S^{+}_n$ be the unicyclic graph obtained from star $S_{n}$ by adding an edge. Then the distance signless Laplacian
 spectral radius of $S^{+}_n$ is the largest root of the following equation:
     $$f(x)=0,$$
 where $f(x)=x^3-(7n-15)x^2+(14n^2-63n+72)x-(8n^3-52n^2+108n-68).$
 \end{lemma}

 \begin{theorem} Let $G$ be a graph of order $n$ with diameter $d\leq 2$. Then
 \begin{equation}
 \partial^Q_1(G)-\partial^L_1(G)\leq \frac{n-6+\sqrt{9n^2-32n+32}}{2}\label{1asd1}
 \end{equation}
 with equality holding if and only if $G\cong S_{n}$\,.
 \end{theorem}

 \begin{proof} If $d=1$, then $G\cong K_n$ and hence
 $$\partial^Q_1(G)-\partial^L_1(G)=n-2<\frac{n-6+\sqrt{9n^2-32n+32}}{2}\,.$$

 Otherwise, $d=2$. Thus we have $D_i=d_i+2(n-1-d_i)=2n-2-d_i$. Therefore $D_1=2n-2-\delta$. First we assume that $\delta\geq 2$. Recall that $\partial^Q_1(G)\leq 2D_{1}$ and $\partial^L_1(G)\geq D_{1}+2$. Then
 $$\partial^Q_1(G)-\partial^L_1(G)\leq 2D_1-D_1-2=D_1-2=2n-4-\delta\leq 2n-6<\frac{n-6+\sqrt{9n^2-32n+32}}{2}\,.$$

Next we assume that $\delta=1$. Since $d=2$, therefore we have $\Delta=n-1$ and $D_1=2n-3$. If $G\cong S_{n}$, then the equality holds
in (\ref{1asd1}). Otherwise, $m\geq n$ and $\Delta=n-1$ and hence $G\supseteqq S^{+}_n$\,. By Theorem \ref{th9}, $\partial^L_1(G)\geq D_1+2=2n-1$.
Then by Lemmas \ref{dv1}, we obtain $\partial^Q_1(G)\leq a$, where $a$ is the largest root of the equation  $f(x)=0,$
where
$$f(x)=x^3-(7n-15)x^2+(14n^2-63n+72)x-(8n^3-52n^2+108n-68).$$
Therefore we have to prove that
$$\partial^Q_1(G)-\partial^L_1(G)\leq a-2n+1\leq \frac{n-6+\sqrt{9n^2-32n+32}}{2}\,,$$
that is,
\begin{equation}
a\leq \frac{5n-8+\sqrt{9n^2-32n+32}}{2}\,.\label{dve1}
\end{equation}
We have $f(b)=12n+4\sqrt{9n^2-32n+32}-20>0$, where
$$b=\displaystyle{\frac{5n-8+\sqrt{9n^2-32n+32}}{2}}\,.$$
Moreover, since $f(n)>0$ and $f(2n)<0$, it follows that $a<b$. This completes the proof.\hspace*{\fill}$\Box$
 \end{proof}

Now, we consider the smallest eigenvalue of the distance signless Laplacian matrix.
\begin{theorem}\label{qs1}
Let $G$ be a connected graph with order $n$ and Wiener index $W$.  Then $$\partial_n^Q(G)\leq \frac{2W}{n}-1.$$
Moreover, If $G\cong K_n,$ then the equality holds.
\end{theorem}

\begin{proof}
Suppose that the transmissions are $D_1\geq D_2\geq \dots\geq D_n$ and $u, v$ are the vertices with $D_u=D_n$ and
$D_v=D_{n-1}$. Giving a vector $X=(x_1,\ldots,x_n)^t$ where
\[x_i=\left\{\begin{array}{ccccccc}
1, \ \ \mbox{if $v_i=v_n$},\\
-1, \ \ \mbox{if $v_i=v_{n-1}$},\\
0, \ \ \mbox{otherwise}.
\end{array}\right.
\]
Then $$\partial_n^Q\leq \frac{X^t\mathcal{Q}X}{X^tX}=\frac{D_n+D_{n-1}-2d(u,v)}{2}\leq \frac{2W}{n}-1.$$
If $G\cong K_n,$ then $\partial_n^Q=n-2=\frac{2W}{n}-1.$\hspace*{\fill}$\Box$
\end{proof}

In view of the proof of Theorem \ref{qs1}, if $D_{u}=D_{v}=D_{n}$, then we obtain the following corollary.

\begin{corollary}\label{cor1}
If $G$ has two vertices with row sum $D_n$, then $\partial_n^Q(G)\leq D_n-1.$
\end{corollary}

Further, for an arbitrary connected graph, one can see that $\partial_n^Q(G)$ is less than $D_{n}$.
\begin{theorem}
Let $G$ be a connected graph with order $n$. Suppose that $D_1\geq D_2\geq \dots\geq D_n$, then $$\partial_n^Q(G)<D_n.$$
\end{theorem}
\begin{proof}
Suppose that $u$ is the vertex with $D_u=D_n$. Giving a vector
$X=(x_1,\ldots,x_n)^t$, where
\[x_i=\left\{\begin{array}{ccccccc}
1, \ \ \mbox{if $v_i=v_n$},\\
0, \ \ \mbox{otherwise}.
\end{array}\right.
\]
Then $$\partial_n^Q\leq  \frac{X^t\mathcal{Q}X}{X^tX}=D_n.$$
If $X$ is the eigenvector of $D^Q$ corresponding to $\partial_n^Q(G)$, then it must orthogonal to eigenvector of $\partial_1^Q(G)$.
It follows that $X$ is not the eigenvector of $D^Q$, then $\partial_n^Q(G)<D_n.$\hspace*{\fill}$\Box$
\end{proof}

We remark that there exist graphs such that $D_n>\partial_n^Q(G)>D_{n}-1$. Recall that the smallest distance signless Laplacian eigenvalue of $S_{n}$ is $\frac{5n-8-\sqrt{9n^2-32n+32}}{2}$~(see \cite{AM}) and the smallest transmission of $S_{n}$ is $n-1$. It is easy to check that $$\frac{5n-8-\sqrt{9n^2-32n+32}}{2}>n-2$$ for $n\geq 3$. That is $D_n>\partial_n^Q(S_{n})>D_{n}-1$. Thus a question naturally arises: which graphs satisfy $D_n>\partial_n^Q(G)>D_{n}-1$.

\section{The distance signless Laplacian spectral radius of unicyclic graph}
\hspace*{\parindent}A \emph{kite} $Ki_{n,3}$ is the unicyclic graph obtained from a cycle $C_{3}$ and a path $P_{n-3}$ by adding an edge between an endpoint of the path and a vertex of $C_{3}$~(see Fig.3). In \cite{AM3}, Aouchiche and Hansen proposed the following conjecture.

\begin{conjecture}
 Let $G$ be a connected unicyclic graph on $n\geq 6$ vertices. Then $\partial^Q_1(G)\leq \partial^Q_1(Ki_{n,3})$ with equality holds if and only if $G\cong Ki_{n,3}$.
\end{conjecture}

 Lin and Zhou \cite{LZ1} proved that the conjecture is correct when the length of the cycle is odd. Here, we completely solve the conjecture. In order to prove this conjecture, we need the following results.

 \begin{lemma} {\bf(\cite{LH1})}\label{edgeDQ}
Let $G_{k,l}(u)$ be the graph defined above. If $k\geq l\geq 2$, then $\partial_{1}^{Q}(G_{k+1,l-1}(u))>\partial_{1}^{Q}(G_{k,l}(u)).$
\end{lemma}

\begin{lemma} {\bf(\cite{LH1})}\label{edgeDQ1}
Let $G_{k,l}(u_{1},v_{1})$ be the graph defined above. If $k\geq l\geq 2$, then $\partial^Q_1(G_{k+1,l-1}(u_{1},v_{1}))>\partial^Q_1(G_{k,l}(u_{1},v_{1})).$
\end{lemma}

\setlength{\unitlength}{0.4pt}
\begin{center}
\begin{picture}(790,333)
\put(65,270){\circle*{4}}
\put(20,310){\circle*{4}}
\qbezier(65,270)(42,290)(20,310)
\put(20,230){\circle*{4}}
\qbezier(65,270)(42,250)(20,230)
\qbezier(20,310)(20,270)(20,230)
\put(111,270){\circle*{4}}
\qbezier(65,270)(88,270)(111,270)
\put(153,270){\circle*{4}}
\qbezier(111,270)(132,270)(153,270)
\put(199,270){\circle*{4}}
\qbezier(153,270)(176,270)(199,270)
\put(246,270){\circle*{4}}
\multiput(199,270)(14.00,0.00){5}{\qbezier(0,0)(0,0)(7.50,0.00)}\qbezier(239.00,270.00)(246,270)(246,270)
\put(294,270){\circle*{4}}
\qbezier(246,270)(270,270)(294,270)
\put(288,65){\circle*{4}}
\put(238,65){\circle*{4}}
\qbezier(288,65)(263,65)(238,65)
\put(189,65){\circle*{4}}
\multiput(238,65)(-14.00,0.00){5}{\qbezier(0,0)(0,0)(-7.50,0.00)}\qbezier(198.00,65.00)(189,65)(189,65)
\put(139,65){\circle*{4}}
\qbezier(189,65)(164,65)(139,65)
\put(89,65){\circle*{4}}
\qbezier(139,65)(114,65)(89,65)
\put(61,88){\circle*{4}}
\qbezier(89,65)(75,77)(61,88)
\put(31,112){\circle*{4}}
\qbezier(61,88)(46,100)(31,112)
\put(59,45){\circle*{4}}
\qbezier(89,65)(74,55)(59,45)
\put(29,25){\circle*{4}}
\qbezier(59,45)(44,35)(29,25)
\put(485,271){\circle*{4}}
\put(528,271){\circle*{4}}
\qbezier(485,271)(506,271)(528,271)
\put(580,271){\circle*{4}}
\qbezier(528,271)(554,271)(580,271)
\put(442,271){\circle*{4}}
\multiput(485,271)(-14.00,0.00){5}{\qbezier(0,0)(0,0)(-7.50,0.00)}\qbezier(445.00,271.00)(442,271)(442,271)
\put(385,271){\circle*{4}}
\qbezier(442,271)(413,271)(385,271)
\put(648,271){\circle*{4}}
\put(698,271){\circle*{4}}
\qbezier(648,271)(673,271)(698,271)
\put(742,271){\circle*{4}}
\multiput(698,271)(14.00,0.00){5}{\qbezier(0,0)(0,0)(7.50,0.00)}\qbezier(738.00,271.00)(742,271)(742,271)
\put(788,271){\circle*{4}}
\qbezier(742,271)(765,271)(788,271)
\put(616,232){\circle*{4}}
\qbezier(616,232)(598,252)(580,271)
\qbezier(616,232)(632,252)(648,271)
\put(614,313){\circle*{4}}
\qbezier(580,271)(597,292)(614,313)
\put(648,272){\circle*{4}}
\qbezier(648,272)(631,293)(614,313)
\put(487,64){\circle*{4}}
\put(531,64){\circle*{4}}
\put(583,64){\circle*{4}}
\put(443,64){\circle*{4}}
\put(388,64){\circle*{4}}
\put(651,64){\circle*{4}}
\put(702,64){\circle*{4}}
\put(742,64){\circle*{4}}
\put(790,64){\circle*{4}}
\put(618,25){\circle*{4}}
\put(618,104){\circle*{4}}
\put(651,64){\circle*{4}}
\qbezier(487,64)(509,64)(531,64)
\qbezier(531,64)(557,64)(583,64)
\multiput(487,64)(-14.00,0.00){5}{\qbezier(0,0)(0,0)(-7.50,0.00)}\qbezier(447.00,64.00)(443,64)(443,64)
\qbezier(443,64)(415,64)(388,64)
\qbezier(651,64)(676,64)(702,64)
\multiput(702,64)(14.00,0.00){5}{\qbezier(0,0)(0,0)(7.50,0.00)}
\qbezier(742,64)(766,64)(790,64)
\qbezier(618,25)(634,45)(651,64)
\qbezier(583,64)(600,84)(618,104)
\qbezier(651,64)(634,84)(618,104)
\qbezier(618,104)(618,65)(618,25)

\put(137,191){$Ki_{n,3}$}
\put(3,322){$u_{2}$}\put(60,284){$u_1$}\put(3,210){$v_{2}$}\put(96,284){$v_1$}\put(132,284){$w_1$}\put(173,284){$w_2$}\put(219,284){$w_{n-5}$}\put(279,284){$w_{n-4}$}

\put(500,191){$U^{4}_{n_{1},n_{2}}$}
\put(1,124){$u_2$}\put(1,12){$v_2$}\put(53,100){$u_1$}\put(53,24){$v_1$}\put(87,76){$w_1$}\put(127,76){$w_2$}\put(208,76){$w_{n-5}$}\put(271,76){$w_{n-4}$}

\put(141,-18){$T^{*}$}
\put(369,284){$v_{n_1}$}\put(415,284){$v_{n_{1}-1}$}\put(512,284){$v_{2}$}\put(609,324){$w_1$}\put(609,210){$w_2$}\put(553,284){$v_{1}$}\put(646,284){$u_1$}\put(681,284){$u_2$}\put(768,284){$u_{n_2}$}

\put(500,-18){$U^{3}_{n_{1},n_{2}}$}
\put(514,76){$v_{2}$}\put(368,76){$v_{n_{1}}$}\put(564,76){$v_1$}\put(418,76){$v_{n_{1}-1}$}\put(606,116){$w_1$}\put(606,2){$w_2$}\put(648,76){$u_1$}\put(688,76){$u_2$}\put(774,76){$u_{n_2}$}

\end{picture}
\vskip 0.4cm Fig. $3$. Graphs $Ki_{n,3}$, $T^{*}$, $U^{3}_{n_{1},n_{2}}$ and $U^{4}_{n_{1},n_{2}}$.
\end{center}

In the next, we consider some special graphs. A tree is \emph{starlike} if exactly one of its vertices has degree larger than 2, and T-\emph{shape} if it is
starlike with maximal degree 3. We will denote by $T(n_1,n_2,n_3)$ the unique T-shape tree such that $T(n_1,n_2,n_3)-v=P_{n_1}\cup P_{n_2}\cup P_{n_3}$,
where $P_{n_i}$ is the path on $n_i$ vertices ($i = 1, 2, 3$), and $v$ the vertex of degree 3. Let $T^{*}$ denote the T-shape tree $T(2,2,n-5)$. Let $U^{4}_{n_1,n_2}$ be a unicyclic graph obtained from $C_{4}$ by attaching two path $P_{n_{1}-1}$, $P_{n_2-1}$ to a pair nonadjacent vertices of $C_{4}$, where $n_{1}\geq n_{2}$. We denote by  $U^{3}_{n_1,n_2}$ the unicyclic graph obtained from $U^{4}_{n_1,n_2}$ by adding the edge $w_{1}w_{2}$ and deleting the edge $v_{1}w_{2}$ (see Fig 3). The following         lemmas compare the distance signless Laplacian spectral radii of the graphs.

\begin{lemma} \label{max}
Let $T^{*}$ be a tree of order $n\geq 7$ which shown in Fig. 3. Then $\partial^Q_1(T^{*})<\partial^Q_1(Ki_{n,3})$.
\end{lemma}

\begin{proof}
Let $X=(x_{u_{1}},x_{u_{2}},x_{v_{1}},x_{v_{2}},x_{w_{1}},\ldots,x_{w_{n-4}})^t$ be a unit eigenvector corresponding to $\partial^Q_1(T^{*})$. It is obvious that $x_{u_1}=x_{v_1}$ and $x_{u_2}=x_{v_2}$. Then
\begin{eqnarray*}
\partial^Q_1(Ki_{n,3})-\partial^Q_1(T^{*})&\geq& X^{t}\mathcal{Q}(Ki_{n,3})X-X^{t}\mathcal{Q}(T^{*})X\\[3mm]
&=&\sum_{i=1}^{n-4}(x_{u_1}+x_{w_i})^{2}+\sum_{i=1}^{n-4}(x_{u_2}+x_{w_i})^{2}+\sum_{i=1}^{n-4}(x_{v_2}+x_{w_i})^{2}+(x_{v_1}+x_{v_2})^{2}\\
&&-(x_{u_1}+x_{v_1})^2-2(x_{u_1}+x_{v_2})^{2}-(x_{u_2}+x_{v_1})^2-3(x_{u_2}+x_{v_2})^2\\[3mm]
&=&\sum_{i=1}^{n-4}(x_{u_1}+x_{w_i})^{2}+2\sum_{i=1}^{n-4}(x_{u_2}+x_{w_i})^{2}-6x_{u_1}^{2}-14x_{u_2}^{2}-4x_{u_1}x_{u_2}.
\end{eqnarray*}
Note that $X$ is a positive vector. If $n\geq 14$, then we obtain
\begin{eqnarray*}
&&\sum_{i=1}^{n-4}(x_{u_1}+x_{w_i})^{2}+2\sum_{i=1}^{n-4}(x_{u_2}+x_{w_i})^{2}-6x_{u_1}^{2}-14x_{u_2}^{2}-4x_{u_1}x_{u_2}\\[3mm]
&&\geq 10x_{u_1}^{2}+20x_{u_2}^2-6x_{u_1}^{2}-14x_{u_2}^{2}-4x_{u_1}x_{u_2}\\[3mm]
&&=4x_{u_1}^{2}+6x_{u_2}^2-4x_{u_1}x_{u_2}\\[3mm]
&&>0.
\end{eqnarray*}
Therefore, $\partial^Q_1(Ki_{n,3})>\partial^Q_1(T^{*})$. When $7\leq n\leq 13$, by calculation, it is evident to see that the lemma holds (see Table 1). This completes the proof.\hspace*{\fill}$\Box$
\end{proof}

\setlength{\tabcolsep}{9pt} 
\renewcommand\arraystretch{1.4}  
\begin{center}
\begin{tabular}{lccccccccc}
\multicolumn{6}{l}{Table 1: The distance signless Laplaican spectral radii of $Ki_{n,3}$ and $T^{*}$.}\\
\hline
&$n=7$&$n=8$&$n=9$&$n=10$&$n=11$&$n=12$&$n=13$\\
\hline
$\partial^Q_1(Ki_{n,3})$ &31.1081&41.6987&53.7733&67.3260&82.3525&98.8494&116.8142\\
$\partial^Q_1(T^{*})$ &29.5507&38.9173&50.0328 &62.7797&77.0989&92.9528&110.3381\\
\hline
\end{tabular}
\end{center}

\begin{lemma} \label{max1}
Let $U^{4}_{n_1,n_2}$ and $U^{3}_{n_1,n_2}$ be the unicyclic graphs of order $n$. If $n\geq 7$, then $\partial^Q_1(U^{4}_{n_1,n_2})<\partial^Q_1(U^{3}_{n_1,n_2})$.
\end{lemma}

\begin{proof}
Let $X=(x_{w_{1}},x_{w_{2}},x_{v_{1}},\ldots,x_{v_{n_1}},x_{u_{1}},\ldots,x_{u_{n_2}})^t$ be a unit eigenvector corresponding to $\partial^Q_1(U^{4}_{n_1,n_2})$. Clearly, $x_{w_1}=x_{w_2}$. Then we obtain
\begin{eqnarray*}
\partial^Q_1(U^{3}_{n_1,n_2})-\partial^Q_1(U^{4}_{n_1,n_2})&\geq& X^{t}\mathcal{Q}(U^{3}_{n_1,n_2})X-X^{t}\mathcal{Q}(U^{4}_{n_1,n_2})X\\[3mm]
&=&\sum_{i=1}^{n_1}(x_{w_2}+x_{v_i})^{2}-(x_{w_1}+x_{w_2})^{2}\\
&=&\sum_{i=1}^{n_1}(x_{w_2}+x_{v_i})^{2}-4x_{w_2}^{2}
\end{eqnarray*}
Since $X$ is a positive vector, if $n_{1}\geq 4$, then $\sum_{i=1}^{n_1}(x_{w_2}+x_{v_i})^{2}-4x_{w_2}^{2}>n_{1}x_{w_2}^{2}-4x_{w_2}^{2}\geq 0$. The remaining graphs are $U^{4}_{3,3}$ and $U^{4}_{3,2}$. By calculation, we have $\partial^Q_1(U^{4}_{3,2})<\partial^Q_1(U^{3}_{3,2})$ and $\partial^Q_1(U^{4}_{3,3})<\partial^Q_1(U^{3}_{3,3})$. Thus we complete the proof.\hspace*{\fill}$\Box$
\end{proof}

Now we are ready to prove the conjecture.

\begin{theorem}
Let $G$ be a connected unicyclic graph on $n\geq 6$ vertices. Then $\partial^Q_1(G)\leq \partial^Q_1(Ki_{n,3})$ with equality holds if and only if $G\cong Ki_{n,3}$.
\end{theorem}

\begin{proof}
Let $G$ be a connected unicyclic graph of order $n\geq 6$. It is easy to check that the result holds when $n=6$. We may assume that $C_{k}$ is the unique cycle in $G$ and $n\geq 7$.

\noindent{\bf Claim 1.} If $k=3$, then $\partial^Q_1(G)\leq \partial^Q_1(Ki_{n,3})$ with equality holds if and only if $G\cong Ki_{n,3}$.

Suppose that $G\ncong Ki_{n,3}$. If there is only one vertex in $C_{3}$ whose degree is at least 3, then it follows from Lemma \ref{edgeDQ} that $\partial^Q_1(G)< \partial^Q_1(Ki_{n,3})$. Otherwise, according to Lemma \ref{edgeDQ1}, we also have $\partial^Q_1(G)< \partial^Q_1(Ki_{n,3})$.

\noindent{\bf Claim 2.} If $k=4$, then $\partial^Q_1(G)<\partial^Q_1(Ki_{n,3})$.

If there are two adjacent vertices in $C_{4}$ whose degree is at least 3, then according to Lemma \ref{edgeDQ1} it follows that $\partial^Q_1(G)< \partial^Q_1(U^{4}_{n_1,n_2})$. Moreover, since $\partial^Q_1(U^{4}_{n_1,n_2})<\partial^Q_1(U^{3}_{n_1,n_2})$, then it follows from claim 1 that $\partial^Q_1(U^{4}_{n_1,n_2})<\partial^Q_1(Ki_{n,3})$. Thus claim holds.

\noindent{\bf Claim 3.} If $4\leq k\leq n-1$, then $\partial^Q_1(G)<\partial^Q_1(Ki_{n,3})$.

Let $v$ be a vertex in $C_{k}$ with $d(v)\geq 3$ and $e$ be an edge of  $C_{k}$ which has the longest distance to $v$. It is easy to see that $G-e$ is a tree. According to Lemma \ref{edgeDQ}, it follows that $\partial^Q_1(G-e)\leq \partial^Q_1(T^{*})$. Consequently, we have $\partial^Q_1(G)<\partial^Q_1(G-e)\leq \partial^Q_1(T^{*})<\partial^Q_1(Ki_{n,3})$.

\noindent{\bf Claim 4.} $\partial^Q_1(C_{n})<\partial^Q_1(Ki_{n,3})$.

Clearly, the distance Laplacian spectral radius of $C_{n}$ is
\[
\partial^Q_1(C_{n})=\left\{\begin{array}{ccccccc}
\frac{n^2}{2}, \ \ \mbox{if $n$ is even},\\
\frac{n^2-1}{2}, \ \ \mbox{if $n$ is odd}.
\end{array}\right.
\]
Note that the Wiener index of $Ki_{n,3}$ is $W(Ki_{n,3})=\frac{1}{6}n(n-1)(n-2)+\frac{1}{2}(n-1)(n-2)+2$. Then we have
$$\partial^Q_1(Ki_{n,3})\geq \frac{4W(Ki_{n,3})}{n}>\frac{2}{3}(n-1)(n-2)+\frac{2}{n}(n-1)(n-2)>\frac{1}{2}n^2\geq \partial^Q_1(C_{n}),$$
when $n\geq 7$. \hspace*{\fill}$\Box$
\end{proof}

\small {

}

\end{document}